\documentclass{article}
\usepackage[paperwidth=7in, paperheight=10in, margin=.875in]{geometry}
\usepackage[backref,colorlinks,linkcolor=red,anchorcolor=green,citecolor=blue]{hyperref}
\usepackage{amsfonts,amssymb}
\usepackage{amsmath}
\usepackage{amsthm}
\usepackage{graphicx}
\usepackage{cite}
\usepackage{enumerate}
\newtheorem{theorem}{Theorem}[section]
\newtheorem{lemma}{Lemma}[section]
\newtheorem{corollary}{Corollary}[section]
\newtheorem{proposition}{Proposition}[section]

\theoremstyle{definition}
\newtheorem{definition}{Definition}[section]
\newtheorem{remark}{Remark}[section]
\newtheorem{example}{Example}[section]

\newcommand*{\id}{{\normalfont\hbox{1\kern-0.15em \vrule width .8pt depth-.5pt}}}

\newcommand{\abs}[1]{\ensuremath{\left| #1 \right|}}
\newcommand{\norm}[1]{\ensuremath{\left\| #1 \right\|}}
\newcommand{\pa}[1]{\ensuremath{\left( #1 \right)}}

\newcommand{\set}[1]{\ensuremath{\left\{ #1 \right\}}}

\allowdisplaybreaks
\begin{document}
\title{Global existence for vector valued fractional reaction-diffusion equations
\thanks{Received date, and accepted date (The correct dates will be entered by the editor).}}

\author{Agust\'in Besteiro\thanks{Instituto de Matem\'atica Luis Santal\'o, CONICET--UBA, 
Ciudad Universitaria, Pabell\'on I (1428) Buenos Aires, Argentina, (abesteiro@dm.uba.ar).}
\and Diego Rial \thanks{Instituto de Matem\'atica Luis Santal\'o, CONICET--UBA and Departamento de Matemática,
Facultad de Ciencias Exactas y Naturales, Universidad de Buenos Aires.
Ciudad Universitaria, Pabellón I (C1428EGA) Buenos Aires, Argentina, (drial@dm.uba.ar).}}

\pagestyle{myheadings} \markboth{Global existence for fractional reaction-diffusion equations}{A. Besteiro and D. Rial} \maketitle

\begin{abstract}
In this paper, we study the initial value problem for infinite dimensional fractional non-autonomous reaction-diffusion equations. 
Applying general time-splitting methods, we prove the existence of solutions globally defined in time using convex sets as invariant 
regions. We expose examples, where biological and pattern formation systems, under suitable assumptions, achieve global existence. 
We also analyze the asymptotic behavior of solutions.
\end{abstract}
     
\providecommand{\keywords}[1]{\textbf{\textit{Keywords---}} #1}
     

\providecommand{\AMS}[1]{\textbf{\textit{AMS code---}} #1}

\section{Introduction}
\label{intro}
In this paper, we prove global existence of solutions for vector valued fractional non-autonomous reaction-diffusion equations. 
That is, we study the non autonomous system
\begin{equation}
\label{eq: reaction-diffusion}
\partial_{t}u + \sigma(-\Delta)^{\beta} u = F(t,u),
\end{equation}
where $u(t,x)\in Z$ for $x\in\mathbb{R}^{n}$, $t>0$, $\sigma \ge 0$ and $0 < \beta \le 1$, $F:\mathbb{R} \times Z \to Z$ 
a continuous map and $Z$ a Banach space. We consider the initial problem $u(x,0)=u_{0}(x)$.

The aim of this paper is to develop a new method to obtain behavioral results on the fractional reaction diffusion equation, 
using recent numerical splitting techniques (\cite{Borgna2015}, \cite{DeLeo2015}) introduced for other purposes. 
The main results of this paper are to obtain general conditions for well posedness of the fractional reaction diffusion equation 
in Banach spaces. 

Fractional reaction-diffusion equations are commonly used on many applications such as biological models, population dynamics models, 
nuclear reactor models, just to name a few (for references to examples see \cite{Baeumer2007}). 
The difference between classical and fractional diffusion is that the classical Laplacian term associated with classical
diffusion implies a Gaussian dispersal kernel in the corresponding equation, which does not represent all possible models in practice. 
The fractional model captures the faster spreading rates and power law invasion profiles observed in many applications. 
The main reason for this behavior is given by the fractional Laplacian, that is described by standard theories of fractional calculus (for a complete survey see \cite{Machado2011}). There are many different equivalent definitions 
of the fractional Laplacian and its behavior is well understood (see \cite{Bucur2016}, \cite{DiNezza2012}, \cite{Kwasnicki2017}, 
\cite{Lischke2018}, \cite{Silvestre2005}, \cite{Pozrikidis2016} and \cite{Landkof1972}). 

The non-autonomous nonlinear reaction diffusion equation dynamics were studied by \cite{Robinson2007} and others, 
analyzing the stability and evolution of the problem.
Global existence in reaction-diffusion equations in bounded sets were studied in the book by Smoller \cite{Smoller1983} and 
in \cite{Chueh1977} where it is considered the $n-$dimensional case with classical diffusion and the intersection of half spaces 
as invariant regions in $\mathbb{R}^{n}$, in which the equation evolves. The case of taking a convex set as an invariant set, 
is considered only when the diffusion coefficients $\sigma_{ij} \in \mathbb{R}^{n \times n}$ is the identity matrix 
(see Corollary 14.8 (b) in \cite{Smoller1983}). Morgan \cite{Morgan1989} considered a similar case in which 
$\sigma_{ij}$ is not the identity matrix, but other conditions are needed over the system to achieve the result. 
These techniques have been used recently many times to obtain global well posedness for different classical diffusion problems
(see \cite{Abdelmalek2014}, \cite{Bendoukha2016}, \cite{Meftah2016}, and \cite{Slavik2017}). 
Motivated by this, in this paper, we study global existence of fractional-diffusion equations using a completely different approach. 
We use time splitting methods in Banach spaces taking closed convex sets as invariant regions. 

As an example, we explore the scalar system where the nonlinearity is given by
$F(u) = (1 + i a)u-(1 + i b)|u|^{2}u$ with $a,b\in\mathbb{R}$ (see \cite{Chueh1977}, \cite{Weimar1994} and \cite{Cartwright2007}).
For particular nonlinearities exact solutions are known, for example, in \cite{Kumar2012} was studied the existence of scalar 
traveling waves for the quadratic, cubic and quartic cases by the tanh method. 
We also explore a FitzHugh Nagumo pattern formation system in $\mathbb{R}^2$ and a population dynamic system in a Banach space. 
In both cases we found an appropriate invariant region that allows us to prove global existence in each case. 
Finally, we also analyze the asymptotic behavior of solutions in the real line.

This paper is organized as follows:
\begin{description}
\item[- Section 2] We introduce the notation and prove some preliminary results concerning the linear and non linear parts 
of the fractional reaction diffusion equation.
\item[- Section 3] We introduce the propagators, allowing us to construct a splitting reaction diffusion equation. 
This is important to build up the linear part.
\item[- Section 4] We obtain several results for finally proving that the "splitting" equation converges to the "original" 
equation. This allow us to study the splitting equation, that is, to study separately the linear and non linear parts, 
in order to obtain interesting results on the original equation.
\item[- Section 5] We prove global well posedness for invariant closed convex sets of a Banach space. 
We prove that the linear and non linear parts of the splitting equation, independently maintain the solution inside the convex set. 
The results from section \ref{section4} extend this result to the "original" equation. 
We give some examples such us the Ginzburg-Landau equation and the Fisher-Kolmogorov equation.
\item[- Subsection 5.1] We expose an interesting example, a population dynamics model, 
where we have a trait variable in a Banach Space. This will show the importance of extending the results to Banach Spaces.
\item[- Subsection 5.2] We generalize the results in the beginning of section \ref{section5}
by proving well posedness for products of Banach Spaces.
\item[- Section 6] We show how powerful splitting methods are, by analyzing a completely different problem, 
the asymptotic behavior of a solution. The strategy is again, to split the linear and nonlinear parts, and analyze them separately, 
for then finally use the results in section \ref{section4} and \ref{section5}.
\end{description}


\section{Notations and Preliminaries.}\label{Sec.2}
 We are interested in continuous functions to vectorial values, that is to say, whose evaluations take values in Banach Spaces. 
 The main reason for this, is to analyze well posedness of population dynamics problems with discrete or continuous traits, 
 that distinguish the population components (see subsection \ref{subsec: population dynamics}).

Let $Z$ be a Banach space, we define $C_{\rm u}(\mathbb{R}^{d},Z)$ as
the set of uniformly continuous and bounded functions on $\mathbb{R}^{d}$ with values in $Z$. Taking the norm
\begin{equation*}
\|u\|_{\infty,Z} = \sup_{x\in\mathbb{R}^{d}}|u(x)|_{Z},
\end{equation*}
$C_{\rm u}(\mathbb{R}^{d},Z)$ is a Banach space. 
It is easy to see that if $g \in L^1(\mathbb{R}^d)$ and $u \in C_u(\mathbb{R}^d,Z)$ the Bochner integral is defined in the following way,
\begin{align*}
\pa{g*u}(x) = \int_{\mathbb{R}^{d}} g(y)u(x-y) dy
\end{align*}
This defines an element of $C_{\rm u}(\mathbb{R}^{d},Z)$ and the linear operator $u\mapsto g*u$ is continuous (see \cite{Cazenave1998}).

The following results show that the operator $-(-\Delta)^{\beta}$ defines a continuous contraction semigroup 
in the Banach space $C_{\rm u}(\mathbb{R}^{d},Z)$.
The following lemma is a consequence of L\'evy--Khintchine formula for infinitely divisible distributions and 
the properties of the Fourier transform.

\begin{lemma}
Let $0<\beta \le 1$ and $g_{\beta}\in C_{0}(\mathbb{R}^{d})$ such that $\hat{g}_{\beta}(\xi) = e^{-|\xi|^{2\beta}}$,
it holds $g_{\beta}$ is positive, invariant under rotations of $\mathbb{R}^{d}$, integrable and
\begin{align*}
\int_{\mathbb{R}^{d}}g_{\beta}(x)dx = 1.
\end{align*}
\end{lemma}

\begin{proof}
The first statement follows from Theorem 14.14 of \cite{Sato1999},
the remaining claims are immediate from the definition of $\hat{g}_{\beta}$.
\end{proof}

Based on the previous lemma, we study Green's function associated to the linear operator $\partial_{t} + \sigma(-\Delta)^{\beta}$.

\begin{proposition}
Let $\sigma > 0$ and $0<\beta \le 1$, the function $G_{\sigma,\beta}$ given by 
\begin{align*}
G_{\sigma,\beta}(t,x) = (\sigma t)^{-\frac{n}{2\beta}} g_{\beta}((\sigma t)^{-\frac{1}{2\beta}} x),
\end{align*}
verifies
\begin{enumerate}[i.]
\item $G_{\sigma,\beta}(.,t)>0$;
\item $G_{\sigma,\beta}(.,t)\in L^{1}(\mathbb{R}^{d})$ and 
\begin{align*}
\int_{\mathbb{R}^{d}}G_{\sigma,\beta}(t,x)dx = 1;
\end{align*}
\item \label{it: semi} $G_{\sigma,\beta}(\cdot,t)*G_{\sigma,\beta}(\cdot,t') = G_{\sigma,\beta}(\cdot,t+t')$, for $t,t'>0$;
\item $\partial_{t}G_{\sigma,\beta} + \sigma (-\Delta)^{\beta} G_{\sigma,\beta} = 0$ for $t>0$.
\end{enumerate}
\end{proposition}
\begin{proof}
The first and second statements are a consequence of the definition of $\hat{g}_{\beta}$.
The third and fourth statements are immediate applying Fourier transform. 
\end{proof}

In the following proposition, we show that the linear operator $-\sigma(-\Delta)^{\beta}$ defines a contraction continuous semigroup in the set $C_{\rm u}(\mathbb{R}^{d},Z)$.

\begin{proposition}
For any $\sigma > 0$ and $0<\beta \le 1$, the map ${\sf S}:\mathbb{R}_{+}\to \mathcal{B}(C_{\rm u}(\mathbb{R}^{d},Z))$ defined by
${\sf S}(t)u = G_{\sigma,\beta}(.,t)*u$ is a continuous contraction semigroup.
\end{proposition}
\begin{proof}
We first prove the semigroup property, which is deduced from \ref{it: semi} of the previous proposition:
\begin{align*}
{\sf S}(t){\sf S}(t')u & = G_{\sigma,\beta}(\cdot,t)*(G_{\sigma,\beta}(\cdot,t')*u) \\
& = (G_{\sigma,\beta}(\cdot,t)*(G_{\sigma,\beta}(\cdot,t'))*u = (G_{\sigma,\beta}(\cdot,t+t'))*u = {\sf S}(t+t')u
\end{align*}

We show that ${\sf S}(t)u$ converges to $u$ for all $u\in C_{\rm u}(\mathbb{R}^{d},Z)$ when $t\to 0$. Indeed, we have for $\delta >0$,
\begin{align*}
|({\sf S}(t)u)(x)-u(x)|_{Z} & \le 
\int_{\mathbb{R}^{n}}G_{\sigma,\beta}(y,t)|u(x-y)-u(x)|_{Z}dy \\
& = \int_{|y|<\delta}G_{\sigma,\beta}(y,t)|u(x-y)-u(x)|_{Z}dy \\
& +\int_{|y|\ge \delta}G_{\sigma,\beta}(y,t)|u(x-y)-u(x)|_{Z}dy.
\end{align*}
The first integral of the right side of the equality can be estimated as follows:
\begin{align*}
\int_{|y|<\delta}G_{\sigma,\beta}(y,t)|u(x-y)-u(x)|_{Z}dy 
& \le \int_{\mathbb{R}^{n}}G_{\sigma,\beta}(y,t)\max_{|y|<\delta}|u(x-y)-u(x)|_{Z}dy \\
& = \max_{|y|<\delta}|u(x-y)-u(x)|_{Z}
\end{align*}
This can be small enough because, $|y|<\delta$ and $u$ is uniformly continuous. For the second term we proceed in the following way,
\begin{align*}
\int_{|y|\ge \delta}G_{\sigma,\beta}(y,t)|u(x-y)-u(x)|_{Z}dy 
& = 2\|u\|_{\infty}(\sigma t)^{-\frac{n}{2\beta}} \int_{|y|\ge \delta} g_{\beta}((\sigma t)^{-\frac{1}{2\beta}} y) dy \\
& = 2\|u\|_{\infty} \int_{|y|\ge \delta(\sigma t)^{-1/(2\beta)}} g_{\beta}(y) dy
\end{align*}
Since $\delta(\sigma t)^{-1/(2\beta)} \to\infty$ when $t\to 0^{+}$ and $g_{\beta}\in L^{1}(\mathbb{R}^{d})$, the right side of the previous equality tends to $0$. 
The next property proves that ${\sf S}$ is well defined, that is ${\sf S}u\in C_{\rm u}(\mathbb{R}^{d},Z)$.
\begin{align*}
|({\sf S}(t)u)(x_{1})-({\sf S}(t)u)(x_{2})|_{Z} & \le 
\int_{\mathbb{R}^{n}}G_{\sigma,\beta}(y,t)|u(x_{1}-y)-u(x_{2}-y)|_{Z}dy \\
& \le \varepsilon\int_{\mathbb{R}^{n}}G_{\sigma,\beta}(y,t)dy = \varepsilon,
\end{align*}
In the last inequality we used that $u$ is uniformly continuous.
And, finally we prove the contraction semigroup property:
\begin{align*}
|({\sf S}(t)u)(x)|_{Z} \le \int_{\mathbb{R}^{n}}G_{\sigma,\beta}(y,t)|u(x-y)|_{Z}dy \le \|u\|_{\infty}.
\end{align*}
\end{proof}
\begin{remark}
If $u\in C_{\rm u}(\mathbb{R}^{d},Z)$ is a constant, then ${\sf S}(t)u=u$.
\end{remark}
 In this paper, we consider integral solutions of the problem \eqref{eq: reaction-diffusion}.
We say that $u\in C([0,T],C_{\rm u}(\mathbb{R}^{d},Z))$ is a mild solution of \eqref{eq: reaction-diffusion}
iff $u$ verifies
\begin{align}
\label{eq: mild solution}
u(t) = {\sf S}(t)u_{0} + \int_{0}^{t} {\sf S}(t-t')F(t',u(t')) dt'.
\end{align}
Since our method to build solutions of \eqref{eq: mild solution} is based on the application of the Lie-Trotter method, it is necessary to study the non-linear problem associated with $ F $. 
We remark that some regularity condition is necessary for convergence, as it is shown in the counterexample given in \cite{Canzi2012}.

Let $F:\mathbb{R}_{+} \times Z \to Z$ be a continuous map, we say that is locally Lipschitz 
in the second variable if, given $R,T>0$ there exists $L = L(R,T)>0$ such that if $t\in [0,T]$ and $z,\tilde{z}\in Z$ 
with $|z|_{Z},|\tilde{z}|_{Z} \le R$, then
\begin{align*}
|F(t,z) - F(t,\tilde{z})|_{Z} \le L |z - \tilde{z}|_{Z}.
\end{align*}
In this case, for any $z_{0}\in Z$
there exists a unique (maximal) solution of the Cauchy problem
\begin{align}
\label{eq: integral equation}
z(t) = z_{0} + \int_{t_{0}}^{t}F(t',z(t'))dt'
\end{align}
defined on $[t_{0},t_{0}+T^{*}(t_{0},z_{0}))$, with $T^{*}(t_{0},z_{0})$ is the maximal time of existence.
It is easy to see that 
there exists a nonincreasing function $\mathcal{T}:\mathbb{R}_{+}^{2} \to \mathbb{R}_{+}$,
such that
\begin{align*}
 \mathcal{T}(T,R)\le \inf\{T^{*}(t_{0},z_{0}):0\le t_{0} \le T, |z_{0}|_{Z} \le R\}.
\end{align*}
Also, one of the following alternatives holds:
\begin{itemize}
\item[-] $T^{*}(t_{0},z_{0}) = \infty$;
\item[-] $T^{*}(t_{0},z_{0}) < \infty$ and $|z(t)|_{Z} \to \infty$ when $t \uparrow t_{0} + T^{*}(t_{0},z_{0})$.
\end{itemize}
We can see that $F:\mathbb{R}_{+}\times C_{\rm u}(\mathbb{R}^{d},Z)\to C_{\rm u}(\mathbb{R}^{d},Z)$,
given by $F(t,u)(x) = F(t,u(x))$ is continuous and locally Lipschitz in the second variable.
Therefore, we can consider problem \eqref{eq: integral equation} in $C_{\rm u}(\mathbb{R}^{d},Z)$.
%
We denote by $\mathsf{N}:\mathbb{R} \times \mathbb{R} \times C_{\rm u}(\mathbb{R}^{d},Z)\to C_{\rm u}(\mathbb{R}^{d},Z)$ the flow 
generated by the integral equation \eqref{eq: integral equation} as $u(t)=\mathsf{N}(t,t_0,u_0)$,
defined for $t_{0} \le t < t_{0} + T^{*}(t_{0},u_{0})$.

The following result relates the solutions of \eqref{eq: integral equation} with the problem \eqref{eq: mild solution} in the case of having constant initial data.

\begin{proposition}
\label{pr: constant}
If $u_{0}$ is a constant function, then $u(t) = \mathsf{N}(t,t_0,u_0)$ is a solution of \eqref{eq: mild solution}.
\end{proposition}
\begin{proof}
Since $u_{0}$ is a constant function, from the uniqueness of the problem \eqref{eq: integral equation}, we have
$u(t)$ is a constant function for any $t>0$ where the solution is defined. Therefore,
\begin{align*}
u(t) & = u_{0} + \int_{0}^{t} F(t',u(t')) dt'
= {\sf S}(t)u_{0} + \int_{0}^{t} {\sf S}(t-t')F(t',u(t')) dt',
\end{align*}
which proves our assertion.
\end{proof}

\begin{theorem}
	\label{th: local existence}
	There exists a  function $T^{*}:C_{\rm u}(\mathbb{R}^{d},Z)\to \mathbb{R}_{+}$ such that
	for $u_{0}\in C_{\rm u}(\mathbb{R}^{d},Z)$, exists a unique $u\in C([0,T^{*}(u_{0})),C_{\rm u}(\mathbb{R}^{d},Z))$ mild solution
	of \eqref{eq: reaction-diffusion} with $u(0) = u_{0}$. Moreover, one of the following alternatives holds:
	\begin{itemize}
		\item $T^{*}(u_{0}) = \infty$;
		\item $T^{*}(u_{0}) < \infty$ and $\lim_{t \uparrow T^{*}(u_{0})}\|u(t)\|_{C_{\rm u}(\mathbb{R}^{d},Z)} = \infty$.
	\end{itemize}
\end{theorem}
\begin{proof}
	See Theorem 4.3.4 in \cite{Cazenave1998}.
\end{proof}

\begin{proposition}
	\label{pr: continuous dependence}
	Under conditions of theorem above, then
	\begin{enumerate}
		\item $T^{*}:C_{\rm u}(\mathbb{R}^{d},Z)\to \mathbb{R}_{+}$ is lower semi-continuous;
		\item If $u_{0,n} \to u_{0}$ in $C_{\rm u}(\mathbb{R}^{d},Z)$ and $0 < T < T^{*}(u_{0})$, then 
		$u_{n} \to u$ in the Banach space $C([0,T],C_{\rm u}(\mathbb{R}^{d},Z))$.
	\end{enumerate}
\end{proposition}
\begin{proof}
	See Proposition 4.3.7 in \cite{Cazenave1998}.
\end{proof}


\section{Propagators}
 \label{section3}

To build the approximate solutions, we decompose the time variable in regular intervals and consider the evolution, in an alternate form, of the linear and non linear problem. To achieve this, we \emph{turn on and off} each term of the equation. The first step, is to consider the abstract linear problem, 
\begin{align*}
& \partial_{t}u - \alpha(t) A u = 0, \\
& u(s) = u_{0},
\end{align*}
where $\alpha(t)\ge 0$ and $A$ is the infinitesimal generator of ${\sf S}$, a strongly continuous semigroup of operators defined in the Banach space $X$. The mild solution of the non autonomous problem can be written as $u(t) = {\sf S}_{\alpha}(t,s) u_{0} = {\sf S}(\tau(t,s)) u_{0}$, where 
$\tau$ is defined by
\begin{align*}
 \tau(t,s) = \int_{s}^{t}\alpha(t')dt'
\end{align*}
Formally, we have $\partial_{t}u = \partial_{t}{\sf S}(\tau(t,s)) u_{0} = \partial_{t}\tau(t,s)A{\sf S}(\tau(t,s))u_{0}$.
To analyze the Lie-Trotter method, we define $\alpha:\mathbb{R} \to \mathbb{R}$ a periodic function of period $1$ as:
\begin{align}
\label{eq: alpha}
\alpha(t) =
\left\{
\begin{array}{cl}
2 & \text{, if } k \le t < k+1/2, \\
0 & \text{, if } k-1/2 \le t < k,
\end{array}
\right.
\end{align}
for $k\in \mathbb{Z}$. 
\label{def: alpha_h}
Given $h>0$, we define the function $\alpha_{h}:\mathbb{R} \to \mathbb{R}$ as $\alpha_{h}(t) = \alpha(t/h)$.
Clearly $0\le \alpha_{h} \le 2$, $\alpha_{h}$ is $h$-periodic and its mean value is $1$.
We consider $\tau_{h}:\mathbb{R}^2 \to \mathbb{R}$ given by
\begin{align*}
\tau_{h}(t,t') = \int_{t'}^{t}\alpha_{h}(t'')dt'',
\end{align*}
\begin{figure}[ht]
\centering
\includegraphics[scale=0.75]{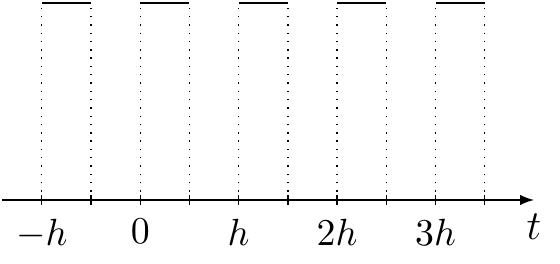}
\caption{\label{fig: alpha} Graph of $\alpha_{h}(t)$.}
\end{figure}
The following results show that ${\sf S}_{\alpha_{h}}$ defines a propagator in $X$.
We also obtain some estimates that we will use in the following section concerning the convergence.
We can prove that 
\begin{lemma}
The map $\tau_{h}$ is continuous in $\mathbb{R}^2$ and satisfies
\begin{enumerate}[i.]
\item $0\le \tau_{h}(t,t')\le 2(t-t')$, if $t'\le t$,
\item $\tau_{h}(t,t')+\tau_{h}(t',t'') = \tau_{h}(t,t'')$, if $t''< t' < t$,
\item $\tau_{h}(t+kh,t'+kh) = \tau_{h}(t,t')$, for $k\in \mathbb{Z}$,
\item $\tau_{h}(t'+kh,t') = kh$, for $k \in \mathbb{Z}$,
\item $\vert(t-t')-\tau_{h}(t,t')\vert \le h$,
\end{enumerate}
\end{lemma}

\begin{proof}
The first statement is a consequence of the inequality $0\le\alpha_{h}\le 2$.
The additivity is immediate from the definition. The third statement is a consequence of the $\alpha_{h}$ periodicity. 
As the mean value of $\alpha_{h}$ is $h$, then $\tau_{h}(t'+h,t') = h$, and using additivity property we have,
\begin{align*}
\tau_{h}(t'+kh,t') = \sum_{j=1}^{k}\tau_{h}(t'+jh,t'+(j-1)h) = kh.
\end{align*}
For the last claim, we consider $t = t' + kh + sh$, with $k \in \mathbb{Z}$ and $0\le s <1$, 
as $\vert 1-\alpha_{h}(t)\vert \le 1$, then
\begin{align*}
\vert(t-t')-\tau_{h}(t,t')\vert & = \vert kh + sh - \tau_h (t'+kh+sh,t')\vert \\ 
& = \vert (kh + sh)-\tau_{h}(t' + kh + sh,t'+kh)-\tau_{h}(t'+kh,t')\vert \\
& = \vert sh - \tau_{h}(t' + kh + sh,t'+kh) \vert \\
& = \left\vert \int_{t'+kh}^{t' + kh + sh} (1 - \alpha_{h}(t''))dt''\right\vert   \\
& \le \int_{t'+kh}^{t' + kh + sh} \vert 1-\alpha_{h}(t'')\vert dt'' \le h,
\end{align*}
that proves the last assertion. 
\end{proof}

We define $\Omega = \{ (t,t') \in \mathbb{R}^2 : 0 \le t' \le t \}$ and the application ${\sf S}_{h}:\Omega \to \mathcal{B}(X)$ defined by
${\sf S}_{h}(t,t') = {\sf S}(\tau_{h}(t,t'))$, from the previous lemma have:
\begin{corollary}
\label{le: Sh}
Let $\sf S:\mathbb{R}_{+}\to \mathcal{B}(X)$ a strongly continuous one-parameter semigroup of operators, we have that ${\sf S}_{h}$ satisfies:
\begin{enumerate}[i.]
\item ${\sf S}_{h}(t,t) = \id$.
\item ${\sf S}_{h}(t,t'') = {\sf S}_{h}(t,t'){\sf S}_{h}(t',t'')$, if $0 \le t''\le t' \le t$.
\item There exist constants $M \ge 1$ and $\omega\in\mathbb{R}$ such that 
$\Vert {\sf S}_{h}(t,t') \Vert_{\mathcal{B}(X)} \le M e^{2\omega (t-t')}$, for $(t,t') \in \Omega$.
\item If $u \in X$, The map $(t,t') \mapsto {\sf S}_{h}(t,t')u$ is continuous.
\item If $u \in D = {\rm Dom}(A)$ and $t'\le t\neq kh/2$ with $k\in \mathbb{Z}$, then the map $t\mapsto {\sf S}_{h}(t,t')u$ is differentiable and we have
\begin{align*}
\partial_t {\sf S}_{h}(t,t')u 
= \left\{
\begin{array}{cc}
2 A {\sf S}_{h}(t,t')u & \text{, if } kh < t < (k+1/2)h, \\
0 & \text{, if } (k-1/2)h < t < kh,
\end{array}
\right.
\end{align*}
\end{enumerate}
\end{corollary}

\begin{figure}[htb]
\centering
\includegraphics[scale=0.6]{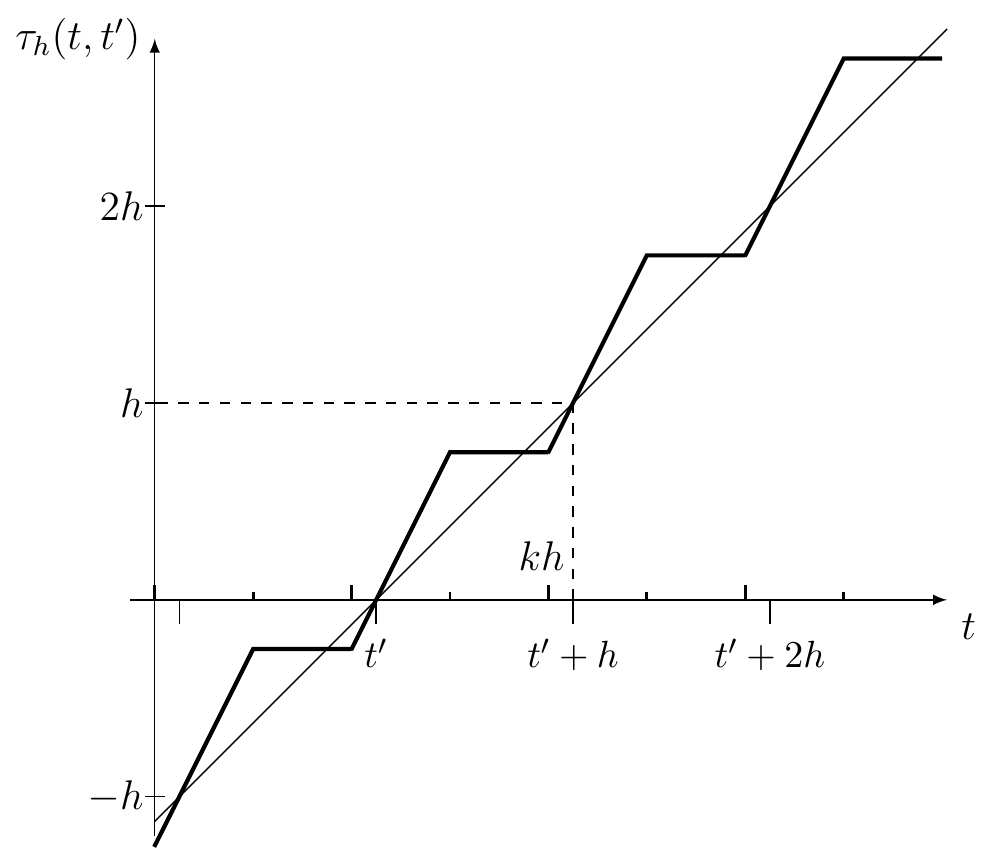}
\caption{\label{fig: tau} Graph of $\tau_{h}(t,t')$, the steps are in the half integers multiples of $h$.}
\end{figure}
\section{Approximate solutions} \label{section4}
In this section we develop the basic tools (Proposition \ref{pr: pasospropagador} and Theorem \ref{th: convergencia}) 
that allow us to obtain some properties of the solutions of the problem \eqref{eq: mild solution} 
from the approximations obtained with the Lie-Trotter method. Theorem \ref{th: convergencia} is an extension of theorem 3.9
 in \cite{DeLeo2015} to the non autonomous case.
We define the system
\begin{align}
\begin{cases}
\label{eq: reaction-diffusionsh}
\partial_t u_h - \alpha_h(t) A u_h = (2-\alpha_h(t)) F(t,u_h), \\
u_h = u_{h,0},
\end{cases}
\end{align}
with $\alpha_h(t)$ as in \eqref{def: alpha_h}, $u\in X$, $t>0$, $F:\mathbb{R}_{+} \times X\to X$ 
is a continuous function and $X$ is a Banach space.
Similarly with define the integral equation:
\begin{align}
\label{eq: u aprox}
u_{h}(t) & = {\sf S}_{h}(t,0)u_{h,0} + \int_{0}^{t} (2-\alpha_{h}(t')) {\sf S}_{h}(t,t') F(t',u_{h}(t')) dt'
\end{align}
\begin{proposition}
Let $u_{h,0} \in {\rm Dom(A)}$,
if $u_{h}$ is solution of the system \eqref{eq: reaction-diffusionsh} then $u_h$ is solution of \eqref{eq: u aprox} for $t \in [0,T]$.
\end{proposition}
\begin{proof}
The procedure is similar to \cite{Cazenave1998}, Lemma 4.1.1.
\end{proof}

\begin{theorem}
\label{th: local existence sh}
There exists a function $T^{*}:X\to \mathbb{R}_{+}$ such that for 
$u_{h,0}\in X$, exists a unique $u_h\in C([0,T^{*}(u_{h,0})),X)$ solution of \eqref{eq: u aprox} with $u_h(0) = u_{h,0}$. 
More over, one of the following alternatives holds:
\begin{itemize}
\item[-] $T^{*}(u_{h,0}) = \infty$;
\item[-] $T^{*}(u_{h,0}) < \infty$ and $\lim_{t \uparrow T^{*}(u_{h,0})}\|u_h(t)\|_{X} = \infty$.
\end{itemize}
\end{theorem}
\begin{proof}
The proof is similar to Theorem 4.3.4 in \cite{Cazenave1998}.
\end{proof}

In the following proposition, we show that the solution of the integral problem \eqref{eq: u aprox} corresponds to the approximations obtained with the Lie-Trotter method.
\begin{proposition}
\label{pr: pasospropagador}
Let $u_{h}$ the solution of \eqref{eq: u aprox}, if $U_{h,k} = u_{h}(kh)$ y $V_{h,k} = u_{h}(kh-h/2)$,
then
\begin{subequations}
\begin{align}
\label{eq: metodo LT lineal}
V_{h,k+1} & = {\sf S}(h)U_{h,k}, \\
\label{eq: metodo LT no lineal}
U_{h,k+1} & = {\sf N}(kh+h,kh+h/2,V_{h,k+1}),
\end{align}
\end{subequations}
where ${\sf N}$ is the flux associated to $2F$, that is $w(t) = {\sf N}(t,t_{0},w_{0})$ where $w$ is the solution of
\begin{align}
\label{eq: integral equation2}
\begin{cases}
\dot{w} = 2F(t,w(t)), \\
w(t_{0}) = w_{0}.
\end{cases}
\end{align}
\end{proposition}
\begin{proof}
For $t_{1}\in (0,t)$ it verifies
\begin{align*}
u_{h}(t) 
& = {\sf S}_{h}(t,t_{1}) u_{h}(t_{1}) + \int_{t_{1}}^{t} (2-\alpha_{h}(t')) {\sf S}_{h}(t,t') F(t',u_{h}(t')) dt'
\end{align*}
using that $t_{1} = kh$ y $t = kh+h/2$, we have
\begin{align*}
V_{h,k+1} & = {\sf S}_{h}(kh+h/2,kh) U_{h,k} 
+ \int_{kh}^{kh+h/2} (2-\alpha_{h}(t')) {\sf S}_{h}(kh+h/2,t') F(t',u_{h}(t')) dt',
\end{align*}
given that $\alpha_{h}(t)=2$ for $t\in [kh,kh+h/2)$, we have $\tau_{h}(kh+h/2,kh) = h$ and therefore
\eqref{eq: metodo LT lineal}.
Similarly, $\alpha_{h}(t)=0$ for $t\in [kh+h/2,kh+h)$, then $\tau_{h}(t,kh+h/2)=0$
and therefore
\begin{align*}
u_{h}(t) = V_{h,k+1} + 2 \int_{kh+h/2}^{t}F(t',u_{h}(t')) dt',
\end{align*}
evaluating in $t=kh+h$, we obtain \eqref{eq: metodo LT no lineal}.
\end{proof}
\begin{theorem}
\label{th: convergencia}
Let $u\in C([0,T^{*}),X)$ the solution of the integral problem \eqref{eq: mild solution}
\begin{align*}
u(t) = {\sf S}(t)u_0 + \int_{0}^{t}{\sf S}(t-t')F(t',u(t')) dt',
\end{align*}
$T\in(0,T^{*})$ y $\varepsilon>0$. There exists $h^{*}>0$ such that if $0<h<h^{*}$, then $u_{h}$ the solution of
\eqref{eq: u aprox} with $u_{h,0} = u_{0}$, is defined in $[0,T]$ and verifies
$\Vert u(t)-u_{h}(t)\Vert_{X} \le \varepsilon$ for $t\in [0,T]$.
\end{theorem}

To prove the theorem, we need two previous lemmas. 
We follow the procedure of Theorem 3.9 in \cite{DeLeo2015} (see also \cite{Borgna2015}).
\begin{lemma}
\label{le: SShf}
Let $f\in C([0,T],X)$, if
\begin{align*}
I_{h}(t,t') = ({\sf S}(t-t') - {\sf S}_{h}(t,t')) f(t'),
\end{align*}
then $\lim\limits_{h\to 0^{+}} \sup\limits_{(t,t')\in\Omega_{T}} \norm{I_{h}(t,t')}_{X} = 0$, 
where $\Omega_{T} = \{ (t,t') \in \mathbb{R}^2 : 0 \le t' \le t \le T\}$.
\end{lemma}
\begin{proof}
Given $\varepsilon>0$, there exists $g\in C([0,T],X)$ such that $g(t)\in D$ for $t\in [0,T]$, 
$Ag\in C([0,T],X)$ and $\max\limits_{t\in [0,T]}\norm{f(t) - g(t)}_{X} < \varepsilon$.
\begin{align}
\begin{split}
\label{eq: Sf-g}
\norm{\pa{{\sf S}\pa{t-t'} - {\sf S}_{h}\pa{t,t'}} \pa{f\pa{t'}-g\pa{t'}}}_{X}
& \le 2M e^{2\omega T}\max\limits_{t\in [0,T]}\norm{f\pa{t} - g\pa{t}}_{X} \\
& \le 2M e^{2\omega T} \varepsilon.
\end{split}
\end{align}
On the other hand, we can write
\begin{align*}
{\sf S}\pa{t-t'} g\pa{t'} & = g\pa{t'} + \int_{0}^{t-t'} {\sf S}\pa{\xi} Ag\pa{t'} d\xi, \\
{\sf S}_{h}\pa{t,t'} g\pa{t'} & = g\pa{t'} + \int_{0}^{\tau_{h}\pa{t,t'}} {\sf S}\pa{\xi} Ag\pa{t'} d\xi,
\end{align*}
subtracting both equations we obtain
\begin{align*}
\pa{{\sf S}\pa{t-t'}-{\sf S}_{h}\pa{t,t'}} g\pa{t'} & = \pm\int_{J\pa{t,t'}} {\sf S}\pa{\xi} Ag\pa{t'} d\xi,
\end{align*}
where $J\pa{t,t'}$ is the interval
$J\pa{t,t'}=\left[\min\set{\pa{t-t'},\tau_{h}\pa{t,t'}},\max\set{\pa{t-t'},\tau_{h}\pa{t,t'}}\right]$, 
then
\begin{align}
\begin{split}
\label{eq: Sg}
\norm{\pa{{\sf S}\pa{t-t'}-{\sf S}_{h}\pa{t,t'}} g\pa{t'}}_{X} & \le 
\abs{\pa{t-t'}-\tau_{h}\pa{t,t'}} \max_{t\in [0,T]}\norm{Ag\pa{t}}_{X} \\
& \le h \max_{t\in [0,T]}\norm{Ag\pa{t}}_{X}.
\end{split}
\end{align}
From the equations \eqref{eq: Sf-g} y \eqref{eq: Sg} we obtain the result.
\end{proof}

\begin{lemma}
\label{le: 1-af}
Let 
$f\in C\pa{\Omega_{T},X}$, with $\Omega_{T}$ as in the previous lemma, if
\begin{align*}
I_{h}\pa{t} = \int_{0}^{t} \pa{\alpha_{h}\pa{t'} - 1} f\pa{t,t'} dt',
\end{align*}
then $\lim\limits_{h\to 0^{+}} \sup\limits_{t \in [0,T]} \norm{I_{h}\pa{t}}_{X} = 0$.
\end{lemma}
\begin{proof}
From the uniform continuity $f$, we can see that exists $\delta>0$ such that if $0 \le t',t''\le t \le T$ and $\abs{t'-t''}<\delta$, 
then $\norm{f\pa{t,t'}-f\pa{t,t''}}_{X} < \varepsilon$.
Let $k = \lfloor t/h \rfloor$, we can write
\begin{align*}
I_{h}\pa{t} = \sum_{j=1}^{k}\int_{\pa{j-1}h}^{jh} \pa{\alpha_{h}\pa{t'} - 1} f\pa{t,t'} dt'
+ \int_{kh}^{t} \pa{\alpha_{h}\pa{t'} - 1} f\pa{t,t'} dt'.
\end{align*}
As the mean value of $\alpha_{h}$ is $1$ in intervals of length $h$, for $f_{j}\in X$ we have
\begin{align}
\label{eq: fj=0}
0 = \int_{\pa{j-1}h}^{jh} \pa{\alpha_{h}\pa{t'} - 1} f_{j} \, dt',
\end{align}
therefore
\begin{align*}
\int_{\pa{j-1}h}^{jh} \pa{\alpha_{h}\pa{t'} - 1} f\pa{t,t'} dt'
= \int_{\pa{j-1}h}^{jh} \pa{\alpha_{h}\pa{t'} - 1} \pa{f\pa{t,t'}-f_{j}} dt'.
\end{align*}
If $h<\delta$ and $f_{j}=f\pa{t,jh}$, then $\norm{f\pa{t,t'}-f_{j}}_{X} < \varepsilon$ for $t\in [\pa{j-1}h,jh]$ and therefore
\begin{align}
\label{eq: f-fj}
\norm{\int_{\pa{j-1}h}^{jh} \pa{\alpha_{h}\pa{t'} - 1} \pa{f\pa{t,t'}-f_{j}} dt'}_{X} \le \varepsilon h.
\end{align}
If $M = \max\limits_{\pa{t,t'}\in\Omega_{T}}\norm{f\pa{t,t'}}_{X}$, then we have
\begin{align}
\label{eq: kht}
\norm{\int_{kh}^{t} \pa{\alpha_{h}\pa{t'} - 1} f\pa{t,t'} dt'}_{X}
\le \int_{kh}^{t}\norm{f\pa{t,t'}}_{X} dt' \le M h.
\end{align}
From \eqref{eq: f-fj}, \eqref{eq: fj=0} y \eqref{eq: kht}, we can obtain
\begin{align*}
\norm{I_{h}\pa{t}}_{X} \le \sum_{j=1}^{k} \varepsilon h + M h \le T \varepsilon + M h,
\end{align*}
from where we get the result.
\end{proof}

\begin{proof}[Proof of Theorem \ref{th: convergencia}]
If $[0,T_{h}^{*})$ is the interval of existence of the integral equation \eqref{eq: u aprox}, 
for $0\le t < \min\set{T,T_{h}^{*}}$ the subtraction $u\pa{t}-u_{h}\pa{t}$ satisfies
\begin{align*}
u\pa{t}-u_{h}\pa{t} = \pa{{\sf S}\pa{t} - {\sf S}_{h}\pa{t,0}} u_0 & + \int_{0}^{t}{\sf S}\pa{t-t'}F\pa{t',u\pa{t'}} dt' \\
& -\int_{0}^{t} \pa{2-\alpha_{h}\pa{t'}} {\sf S}_{h}\pa{t,t'} F\pa{t',u_{h}\pa{t'}} dt'.
\end{align*}
If we define
\begin{align*}
I_{2,h}\pa{t} & = \int_{0}^{t} \pa{2-\alpha_{h}\pa{t'}} \pa{{\sf S}\pa{t-t'} - {\sf S}_{h}\pa{t,t'}}F\pa{t',u\pa{t'}} dt', \\
I_{3,h}\pa{t} & = \int_{0}^{t} \pa{\alpha_{h}\pa{t'} - 1} {\sf S}\pa{t-t'} F\pa{t',u\pa{t'}} dt', \\
\end{align*}
then
\begin{align}
\begin{split}
\label{eq: u-uh}
u\pa{t}-u_{h}\pa{t} & = I_{1,h}\pa{t} + I_{2,h}\pa{t} + I_{3,h}\pa{t} \\
& + \int_{0}^{t} \pa{2-\alpha_{h}\pa{t'}} {\sf S}_{h}\pa{t,t'} \pa{F\pa{t',u\pa{t'}} - F\pa{t',u_{h}\pa{t'}}} dt'.
\end{split}
\end{align}
Using the Lemma \ref{le: SShf}, using $f\pa{t}=u_0$, 
we can see that $\lim\limits_{h\to 0}\sup\limits_{t\in [0,T]}\norm{I_{1,h}\pa{t}}_{X}=0$.
Given that,
\begin{align*}
\norm{I_{2,h}\pa{t}}_{X} & \le 2 \int_{0}^{t} \norm{\pa{{\sf S}\pa{t-t'} - {\sf S}_{h}\pa{t,t'}}F\pa{t',u\pa{t'}}}_{X} dt'\\
& \le 2 T \sup\limits_{\pa{t,t'}\in\Omega_{T}}\norm{\pa{{\sf S}\pa{t-t'} - {\sf S}_{h}\pa{t,t'}}F\pa{t',u\pa{t'}}}_{X},
\end{align*}

using once again the Lemma \ref{le: SShf} for $f\pa{t} = F\pa{t,u\pa{t}}$, 
we obtain $\lim\limits_{h\to 0}\sup\limits_{t\in [0,T]}\norm{I_{2,h}\pa{t}}_{X}=0$.
The map $f\pa{t,t'} = {\sf S}\pa{t-t'} F\pa{t',u\pa{t'}}$ is continuous in $\Omega_{T}$, therefore
from Lemma \ref{le: SShf}, we can deduce $\lim\limits_{h\to 0}\sup\limits_{t\in [0,T]}\norm{I_{3,h}\pa{t}}_{X}=0$.

We consider $R = \max\limits_{t\in [0,T]}\norm{u\pa{t}}_{X} + \varepsilon$ and $L$ the Lipschitz constant
of $F$ for $B_{R}\pa{0}\subset X$, if we define
\begin{align*}
J_{\varepsilon} = \set{0\le t < \min\set{T,T_{h}^{*}}: \norm{u_{h}\pa{t'}}_{X} < R + \varepsilon, 0\le t' \le t},
\end{align*}
from the estimate \eqref{eq: u-uh} we obtain for $t\in J_{\varepsilon}$:
\begin{align*}
\norm{u\pa{t}-u_{h}\pa{t}}_{X} & \le \norm{I_{1,h}\pa{t}}_{X} + \norm{I_{2,h}\pa{t}}_{X} + \norm{I_{3,h}\pa{t}}_{X} \\
& + 2 M e^{2\omega T} L \int_{0}^{t} \norm{u\pa{t'} - u_{h}\pa{t'}}_{X} dt',
\end{align*}
and from Gronwall's lemma
\begin{align*}
\norm{u\pa{t}-u_{h}\pa{t}}_{X} & \le e^{C T}\pa{\sup_{t\in [0,T]}\norm{I_{1,h}\pa{t}}_{X} 
+ \sup_{t\in [0,T]}\norm{I_{2,h}\pa{t}}_{X} + \sup_{t\in [0,T]}\norm{I_{3,h}\pa{t}}_{X}}
\end{align*}
where $C = 2 M e^{2\omega T} L$. Taking $h^{*}>0$ small enough,
we have that $\norm{u\pa{t}-u_{h}\pa{t}}_{X}<\varepsilon/2$ for $t\in J_{\varepsilon}$ y $0<h<h^{*}$.
Therefore $\sup J_{\varepsilon} = \min\set{T,T_{h}^{*}}$, but as $\norm{u_{h}\pa{t}} \le R + \varepsilon < \infty$,
it verifies $T < T_{h}^{*}$, that proves the theorem.
\end{proof}



\section{Global well posedness of the Cauchy problem} 
\label{section5}
In this section, we analyze the well posedness for the problem \eqref{eq: integral equation} for different interesting cases. 
The local case can be analyzed using standard methods, so we refer the reader to the bibliography.
We address the global problem, $t\in[0,\infty)$, by the notion of positively invariant convex families.
For classical diffusion ($\beta = 1$), similar ideas can be found in chapter 14 of \cite{Smoller1983}.
But this method presents two problems, the operator must be a differential elliptic operator and $u(x)$ belongs to a space of finite dimension, in order to use some maximum principle.
Both difficulties are overcome considering the Lie-Trotter approximations and then passing to the limit.
We take advantage of this, to study the evolution of a population model, where individuals have a characteristic trait that differentiates them.
In \cite{Arnold2012} the existence of stationary solutions is studied, for a scalar characteristic trait.
In order not to limit a priori the possibilities of modeling this problem, we consider the abstract case, where the characteristic trait is an element in a measure space.
\label{sec: Well posedness}

\begin{definition}
Let $\{K(t)\}_{t\in\mathbb{R}_{+}}$ be a family of closed sets of $Z$, 
we say that $\{K(t)\}_{t\in\mathbb{R}_{+}}$ is positively $F$--invariant if
for any $t_{0}\in\mathbb{R}_{+}$, $z_{0}\in K(t_{0})$, the solution $z$ of \eqref{eq: integral equation} verifies
$z(t)\in K(t)$ for $t\in [t_{0},t_{0}+T^{*}(t_{0},z_{0}))$. The family $\{K(t)\}_{t\in\mathbb{R}_{+}}$
is increasing if $K(t')\subseteq K(t)$ for $0\le t'\le t$.
\end{definition}
\begin{example}
Let $a,b\in C(\mathbb{R}_{+})$ be positive continuous functions defined on $\mathbb{R}_{+}$ such that 
$|F(t,z)|_{Z}\le a(t)+b(t)|z|_{Z}$
for $(t,z)\in \mathbb{R}_{+}\times Z$, we claim that the family of closed balls 
given by $B(t)=\{z\in Z:|z|_{Z}\le \lambda(t) \}$, with
\begin{align*}
\lambda(t) = \left(\lambda_{0} + \int_{0}^{t} a(t') dt'\right)\exp\left(\int_{0}^{t} b(t') dt'\right),
\end{align*}
is increasing and positively $F$--invariant family of (convex) closed sets.
Indeed, since $\lambda(t)$ is a increasing function, it is clear that $\{B(t)\}_{t\in\mathbb{R}_{+}}$ is increasing family.
Let $z_{0}\in B(t_{0})$, from \eqref{eq: integral equation} we obtain
\begin{align*}
 |z(t)|_{Z} \le |z_{0}|_{Z} + \int_{t_{0}}^{t} |F(t',z(t'))|_{Z} dt' 
 \le \lambda(t_{0}) + \int_{t_{0}}^{t} \left(a(t')+b(t')|z|_{Z}\right) dt'.
\end{align*}
From Gronwall's Lemma, we have
\begin{align*}
 |z(t)|_{Z} \le \left(\lambda(t_{0}) + \int_{t_{0}}^{t} a(t') dt'\right) \exp\left(\int_{t_{0}}^{t} b(t')|z|_{Z}\right) 
 \le \lambda(t),
\end{align*}
which implies $z(t)\in B(t)$.
\end{example}
\begin{lemma}
Let $\{K(t)\}_{t\in\mathbb{R}_{+}}$ be a family of closed sets of $Z$. If $\{K(t)\}_{t\in\mathbb{R}_{+}}$ is positively $F$--invariant
and $F$ is autonomous, 
then for any $z_{0}\in K(t_{0})$ and $0<h<T^{*}(z_{0})$, the solution $z$ of \eqref{eq: integral equation2}
with initial condition $z(t_{0}+h/2) = z_{0}$, verifies $z(t_{0}+h)\in K(t_{0}+h)$.
\end{lemma}
\begin{proof}
Let $w(t) = z((t+t_{0}+h)/2)$, we have
\begin{align*}
 w(t_{0}+h) & = z(t_{0}+h) = z_{0} + \int_{t_{0}+h/2}^{t_{0}+h}2F(z(t')) dt' \\
 & = z_{0} + \int_{t_{0}}^{t_{0}+h}F(z((t+t_{0}+h)/2)) dt = z_{0} + \int_{t_{0}}^{t_{0}+h}F(w(t)) dt.
\end{align*}
Using $\{K(t)\}_{t\in\mathbb{R}_{+}}$ is positively $F$--invariant, we have $w(t_{0}+h)\in K(t_{0}+h)$.
\end{proof}
\begin{lemma}
Let $\{K(t)\}_{t\in\mathbb{R}_{+}}$ be a increasing family of closed sets of $Z$ such that $\{K(t)\}_{t\in\mathbb{R}_{+}}$
is positively $2F$--invariant
then for any $z_{0}\in K(t_{0})$ and $0<h<T^{*}(z_{0})$, the solution $z$ of \eqref{eq: integral equation2}
with initial condition $z(t_{0}+h/2) = z_{0}$, verifies $z(t_{0}+h)\in K(t_{0}+h)$.
\end{lemma}
\begin{proof}
Since $z_{0}\in K(t_{0})\subseteq K(t_{0}+h/2)$ and $\{K(t)\}_{t\in\mathbb{R}_{+}}$ is positively $2F$--invariant,
the result follows.
\end{proof}
\begin{corollary}
\label{co: Nhh2}
Let $F:\mathbb{R}_{+} \times Z \to Z$ be a continuous map locally Lipschitz 
in the second variable and $\{K(t)\}_{t\in\mathbb{R}_{+}}$ is a family of closed sets of $Z$.
If one of the following conditions holds
\begin{itemize}
 \item[-] $F$ is autonomous and $\{K(t)\}_{t\in\mathbb{R}_{+}}$ is positively $F$--invariant,
 \item[-] $\{K(t)\}_{t\in\mathbb{R}_{+}}$ is increasing and positively $2F$--invariant,
\end{itemize}
then for any $u_{0}\in C_{\rm u}(\mathbb{R}^{d},K(t_{0}))$ and $0<h<T^{*}(t_{0},u_{0})$ it is verified
\begin{align*}
 {\sf N}(t_{0}+h,t_{0}+h/2,u_{0}) \in C_{\rm u}(\mathbb{R}^{d},K(t_{0}+h)).
\end{align*}
\end{corollary}

\begin{lemma}
\label{le: G*u}
Let $\sigma \ge 0$, $0 < \beta \le 1$ and let $K$ be a closed convex set of $Z$,
for any $t>0$ and $u\in C_{u}(\mathbb{R}^{d},K)$, it holds ${\sf S}(t)u \in C_{u}(\mathbb{R}^{d},K)$.
\end{lemma}
\begin{proof}
Suppose, contrary to our claim, that the assertion of the lemma is false. 
Then, there exists $(t,x)\in\mathbb{R}^{d}\times \mathbb{R}_{+}$
such that $v = ({\sf S}(t) u)(x)\notin K$. Using Hahn-Banach separation theorem, we take a separating hyperplane; 
i.e., $\omega\in Z^{*}$ and $\lambda\in\mathbb{R}$ verifying
$\langle\omega,z\rangle\le\lambda$ for any $z\in K$ and $\langle\omega,v\rangle > \lambda$, but
\begin{align*}
\left\langle\omega, v \right\rangle 
& = \int_{\mathbb{R}^{d}}G_{\sigma,\beta}(x-y,t)\left\langle\omega,u(y)\right\rangle dy
\le \lambda \int_{\mathbb{R}^{d}}G_{\sigma,\beta}(x-y,t)dy = \lambda,
\end{align*}
a contradiction.
\end{proof}
\begin{proposition}
\label{pr: u in CK}
Let $F$ and $\{K(t)\}_{t\in\mathbb{R}_{+}}$ be as Corollary \ref{co: Nhh2}.
If $K(t)$ is convex for $t\ge0$, then $u(t)\in C_{u}(\mathbb{R}^{d},K(t))$ for any $u_{0}\in C_{u}(\mathbb{R}^{d},K(0))$ and $t \in (0,T^{*}(u_{0}))$,
where $u$ is the solution of \eqref{eq: mild solution}.
\end{proposition}
\begin{proof}
For $t \in [0, T^{*}(u_{0}))$ and $n\in\mathbb{N}$, Let $h = t/n$ and 
$\{U_{h,k}\}_{0\le k\le n},\{V_{h,k}\}_{1\le k\le n}$ be the sequences given by $U_{h,0} = u_{0}$,
\begin{subequations}
\label{eq: L-T}
\begin{align}
V_{h,k+1} & = {\sf S}(h) U_{h,k}, \\
U_{h,k+1} & = \mathsf{N}(kh+h,kh+h/2,V_{h,k+1}), \quad k=0,\dots,n-1.
\end{align}
\end{subequations}
From Proposition \ref{pr: pasospropagador} and Theorem \ref{th: convergencia}, 
it may be concluded that $U_{h,k}$ is defined and $\|u(t)-U_{h,n}\|_{\infty,Z}\to 0$ when $n\to\infty$.
Since $K(t)$ is a closed set, it is suffices to prove that $U_{h,n}\in K(t)$.
We claim that $U_{h,k}\in C_{u}(\mathbb{R}^{d},K(kh))$, the proof is by induction on $k$. 
If $U_{h,k}\in C_{u}(\mathbb{R}^{d},K(kh))$, as $K(kh)$ is convex, Lemma \ref{le: G*u} implies
$V_{h,k+1}\in C_{u}(\mathbb{R}^{d},K(kh))$.
From Corollary \ref{co: Nhh2}, $U_{h,k+1}\in C_{u}(\mathbb{R}^{d},K((k+1)h))$
and our claim follows.
\end{proof}
\begin{theorem}
\label{th: global existence}
Let $\{K(t)\}_{t\in\mathbb{R}_{+}}$ be a family of bounded convex closed sets of $Z$. 
Suppose that $F$ and $\{K(t)\}_{t\in\mathbb{R}_{+}}$ satisfy the hypothesis of Corollary \ref{co: Nhh2} and
for any $T>0$, it is verified $M(T) = \sup\{|z|_{Z}: z\in K(t), t\in[0,T] \}<\infty$,
then for any $u_{0}\in C_{u}(\mathbb{R}^{d},K(0))$, it holds $T^{*}(u_{0}) = \infty$ and $u(t)\in C_{u}(\mathbb{R}^{d},K(t))$ for $t>0$.
\end{theorem}
\begin{proof}
From Proposition \ref{pr: u in CK}, we have $u(t)\in C_{u}(\mathbb{R}^{d},K(t))$ for $t \in (0,T^{*}(u_{0}))$.
Suppose $T^{*}(u_{0})<\infty$, then $\lim_{t\to T^{*}(u_{0})}\|u(t)\|_{\infty,Z} = \infty$. But
$\|u(t)\|_{\infty,Z} \le M(T^{*}(u_{0})) < \infty$, a contradiction.
\end{proof}
\begin{example}[Ginzburg-Landau equation]
The Ginzburg-Landau equation is given by \eqref{eq: reaction-diffusion},
where $\beta = 1$, $\sigma>0$
and $F(u) = (1 + i a)u-(1 + i b)|u|^{2}u$ with $a,b\in\mathbb{R}$ (see \cite{Cartwright2007}, \cite{Chueh1977} and \cite{Weimar1994}).
In general, we consider $F(u) = f_{\rm R}(|u|^{2})u + i f_{\rm I}(|u|^{2})u$, where $f_{\rm R},f_{\rm I} :\mathbb{R}_{+}\to\mathbb{R}$
are smooth functions. If $f_{\rm R}(\eta) \le 0$ for $\eta > 0$, then $K = B(0,\eta)$
is a bounded convex positively $F$--invariant set of $\mathbb{C}$.
For $0 < \beta \le 1$, from Theorem \ref{th: global existence}, we obtain that the fractional Ginzburg-Landau equation
is globally well posed for $u_{0}\in C_{u}(\mathbb{R}^{d},K)$. 
\end{example}
\begin{example}[Fisher--Kolmogorov equation]
Fisher \cite{Fisher1937} and Kolmogorov et al \cite{Kolmogorov1989}
introduced a classical model to describe the 
propagation of an advantageous gene in a one-dimensional habitat. 
We consider the generalized non-linear reaction-diffusion equation
\begin{align*}
\partial_{t}u + \sigma (-\Delta)^{\beta}u =\chi u (1-u),
\end{align*}
where $u$ is the chemical concentration, $\sigma$ is the diffusion coefficient and the positive constant $\chi$
represents the growth rate of the chemical reaction. Since then a 
great deal of work has been carried out to extend their 
model to take into account the other biological, chemical 
and physical factors. This equation is also used in 
flame propagation (\cite{Frank-Kamenetskii1955}), nuclear reactor theory (\cite{Canosa1969}), autocatalytic chemical reactions
(\cite{Cohen1971} and \cite{Fife1977}), logistic growth models (\cite{Murray1977}) and neurophysiology (\cite{Tuckwell1988}). 
Consider $b_{0}>1$ and $K(t)=[0,b(t)]$, with
\begin{align*}
b(t) = \frac{b_{0}e^{\chi t}}{1+b_{0}(e^{\chi t}-1)},
\end{align*}
we can see that $\{K(t)\}_{t\in\mathbb{R}_{+}}$ is a family of compact
intervals, positively $F$--invariant for $F(z)=\chi z (1-z)$. 
In particular,
for any $u_{0}\in C_{\rm u}(\mathbb{R})$ with 
$u_{0}(x) \ge 0$, taking $b_{0}=\sup_{x\in\mathbb{R}^{d}}u(x)$,
we can see that $T^{*}(u_{0})=\infty$ and 
$\lim\sup_{t\to\infty}|u(t,x)|\le 1$ for any $x\in\mathbb{R}^{d}$.
In the case $0<a_{0}=\inf_{x\in\mathbb{R}^{d}}u(x)<1$, we have that
$K(t) =[a(t),b(t)]$ with
\begin{align*}
a(t) = \frac{a_{0}e^{\chi t}}{1+a_{0}(e^{\chi t}-1)},
\end{align*}
is $F$--positive. Therefore, $\lim_{t\to\infty}\|u(t)-1\|_{\infty}=0$.
\end{example}
\subsection{Population dynamics with a continuous trait}
\label{subsec: population dynamics}
In \cite{Arnold2012}, Arnold et al. consider a model
of population dynamics in which the population is structured with respect to the space variable $x$ and a trait variable denoted by $\theta$.
The distribution function $u(t,x,\theta)\ge 0$ denote the number density
of individuals at time $t\in\mathbb{R}_{+}$, position $x\in\mathbb{R}^{d}$, and whose trait is $\theta \in \Theta$.
The evolution of $u$ is governed by an integro-PDE model of reaction-diffusion type
in infinite (continuous) dimension in which selection, mutations, competition, and
migrations are taken into account.
The modeling assumptions are the following: migration is described by a (normal or anomalous) diffusion operator $-\sigma(-\Delta)^{\beta}$; mutations are described by a linear kernel $M(\theta,\vartheta)$ which is related to the probability that individuals with trait $\vartheta$ have offsprings with trait $\theta$;
selection is implemented in the model, thanks to a fitness function
$k$ which may depend on trait $\theta$; 
finally a logistic term involving a kernel $C(\theta,\vartheta)$ models the competition (felt by individuals of trait $\theta$) due to individuals of trait $\vartheta$.
Under those assumptions, the evolution of the population is
governed by the following integro-PDE:
\begin{align}
\label{eq: in L1(X)}
\partial_{t}u + \sigma(-\Delta_{x})^{\beta}u = F(t,u(t))
\end{align}
with initial condition $u(0)=u_{0}$. The map $F$ is given by
\begin{align*}
F(t,z)(\theta) = k(t,\theta)z(\theta) & + \int_{\Theta}M(t,\theta,\vartheta )z(\vartheta ) d\mu(\vartheta) \\
& -\left(\int_{\Theta}C(t,\theta,\vartheta )z(\vartheta) d\mu(\vartheta) \right)z(\theta),
\end{align*}
Let $\Theta$ be a compact Hausdorff space, $\mathcal{B}$ the $\sigma$-algebra of Borel sets and $\mu$ a regular Borel probability,
we set the problem on $C_{\rm u}(\mathbb{R}^{d},Z)$, with $Z = L^{1}(\Theta,\mathcal{B},\mu)$.
Following \cite{Arnold2012}, we assume $k\in C(\mathbb{R}_{+}\times\Theta )$, 
$M,C\in C(\mathbb{R}_{+}\times\Theta\times\Theta )$ verifying $M \ge 0$ and $C>0$. For any $T>0$, we define
\begin{align*}
 \|k\|_{T,\infty} & = \max\{|k(t,\theta)|: (t,\theta)\in[0,T]\times\Theta\}, \\
 \|M\|_{T,\infty} & = \max\{M(t,\theta,\vartheta): (t,\theta,\vartheta)\in[0,T]\times\Theta\times\Theta\}, \\
 \|C\|_{T,\infty} & = \max\{C(t,\theta,\vartheta): (t,\theta,\vartheta)\in[0,T]\times\Theta\times\Theta\}.
\end{align*}
Also, we need
\begin{subequations}
\label{eq: kc}
\begin{align}
\label{eq: k}
k_{+}(t) & = \max\left\{k(t',\theta) + \int_{\Theta} M(t',\vartheta,\theta )d\mu(\vartheta)
: (t',\theta)\in [0,t]\times\Theta\right\},\\
c_{-}(t) & = \min\left\{C(t',\theta,\vartheta):(t',\theta,\vartheta) \in [0,t]\times\Theta\times\Theta \right\}.
\label{eq: c}
\end{align}
\end{subequations}
We assume that $c_{-}(t) > 0$ for $t>0$, the lower bound for $C$ means that all individuals are in competition.
To obtain well-posedness of \eqref{eq: in L1(X)}, we prove some previous results.
\begin{lemma}
The map $F:\mathbb{R}_{+}\times Z \to Z$ is continuous and locally Lipschitz in the second variable.
\end{lemma}
\begin{proof}
Let $R,T>0$ and $z,\tilde{z}\in Z$ with $|z|_{Z},|\tilde{z}|_{Z}\le R$, we have
\begin{align*}
|F(t,z)-F(t,\tilde{z})|_{Z} & \le 
\int_{\Theta}|k(t,\theta)||z(\theta)-\tilde{z}(\theta)|d\mu(\theta)\\
& + \int_{\Theta\times \Theta}M(t,\theta,\vartheta )|z(\vartheta )-\tilde{z}(\vartheta )|d\mu(\vartheta) d\mu(\theta) \\
& + \int_{\Theta\times \Theta}C(t,\theta,\vartheta )|z(\vartheta )||z(\theta)-\tilde{z}(\theta)|
d\mu(\vartheta) d\mu(\theta) \\
& + \int_{\Theta\times \Theta}C(t,\theta,\vartheta )|\tilde{z}(\theta)||z(\vartheta)-\tilde{z}(\vartheta)|
d\mu(\vartheta) d\mu(\theta).
\end{align*}
Using $k,M,C$ are bounded for $t\in[0,T]$ and $\theta,\vartheta\in\Theta$, we get
\begin{align*}
|F(t,z)-F(t,\tilde{z})|_{Z} & \le 
\left(\|k\|_{T,\infty}+\|M\|_{T,\infty} + 2\|C\|_{T,\infty} R \right) |z-\tilde{z}|_{Z}.
\end{align*}
Let $(t_{n},z_{n})\to (t,z) \in [0,T]\times\Theta$, we can see that
\begin{align*}
|F(t,z)-F(t_{n},z_{n})|_{Z} & \le |F(t,z)-F(t_{n},z)|_{Z} + |F(t_{n},z)-F(t_{n},z_{n})|_{Z} \\
& \le |F(t,z)-F(t_{n},z)|_{Z} + L\pa{R,T} |z - z_{n}|_{Z},
\end{align*}
using that
\begin{align*}
|F(t,z)-F(t_{n},z)|_{Z} & \le \int_{\Theta}|k(t,\theta) - k(t_{n},\theta)||z(\theta)|d\mu(\theta)\\
 & + \int_{\Theta\times \Theta}|M(t,\theta,\vartheta ) - M(t_{n},\theta,\vartheta)||z(\vartheta )|d\mu(\vartheta) d\mu(\theta) \\
 & + \int_{\Theta\times \Theta}|C(t,\theta,\vartheta) - C(t_{n},\theta,\vartheta)||z(\vartheta )||z(\theta)|
d\mu(\vartheta) d\mu(\theta),
\end{align*}
from uniform continuity of $k,M,C$, we obtain $F(t_{n},z) \to F(t,z)$ in $Z$, which complete the proof.
\end{proof}
We have the same result for continuous functions:
\begin{lemma}
The map $F:\mathbb{R}_{+}\times C( \Theta )\to C( \Theta )$ is continuous and locally Lipschitz in the second variable.
\end{lemma}
\begin{proof}
The proof is similar to the above lemma.
\end{proof}
The nonnegativity of density $z(t,\theta)$ is established by the next proposition
(and corollary below).
\begin{proposition}
\label{pr: u>0}
Let $z$ be the solution of \eqref{eq: integral equation} with $z(t_{0})=z_{0}\in C( \Theta )$.
If $z(t_{0})>0$ then $z(t)>0$ for any $t\in [t_{0},t_{0} + T^{*}(t_{0},z_{0}))$.
\end{proposition}
\begin{proof}
Let $0 < T < T^{*}(t_{0},z_{0})$, for any $(t,\theta)\in [t_{0},t_{0} + T]\times\Theta$, we define 
\begin{align*}
g(t,\theta) & = \int_{\Theta}M(t,\theta,\vartheta )z(t,\vartheta) d\mu(\vartheta) \\
a(t,\theta) & =\int_{\Theta}C(t,\theta,\vartheta )z(t,\vartheta) d\mu(\vartheta) 
\end{align*}
then $g(.,\theta),a(.,\theta)$ are continuous, the solution verifies $z(.,\theta)\in C^{1}([t_{0},t_{0} + T^{*}(t_{0},z_{0})))$ and
\begin{align*}
\begin{cases}
\partial_{t} z(t,\theta) = (k(t,\theta)-a(t,\theta))z(t,\theta) + g(t,\theta),\\
z(t_{0},\theta)=z_{0}(\theta).
\end{cases}
\end{align*}
Then
\begin{align}
\label{eq: z(t)}
z(t,\theta) = e^{A(t,t_{0},\theta)}z_{0}(\theta) + \int_{t_{0}}^{t}e^{A(t,t',\theta)}g(t',\theta) dt',
\end{align}
where
\begin{align*}
A(t,t',\theta) = \int_{t'}^{t}k(t'',\theta) - a(t'',\theta) dt'' .
\end{align*}
Let $t_{*} = \sup\{t\in [t_{0},t_{0} + T] : \min\limits_{[t_{0},t]\times\Theta} z(t,\theta)>0\}$.
Suppose $t_{*} < t_{0} +T$, there exists $\theta_{*} \in \Theta $
with $z(\theta_{*},t_{*}) = 0$.
But from \eqref{eq: z(t)}, we have
\begin{align*} z(t_{*},\theta_{*})=e^{A(t_{*},t_{0},\theta_{*})}z_{0}(\theta_{*}) 
+ \int_{t_{0}}^{t_{*}}e^{A(t_{*},t',\theta_{*})}g(t',\theta_{*}) dt'>0, 
\end{align*}
a contradiction. Since $T$ is arbitrary, we obtain the result.
\end{proof}
\begin{corollary}
\label{co: uge0}
Let $z$ be the solution of \eqref{eq: integral equation} with $z(t_{0})=z_{0}\in C( \Theta )$.
If $z_{0} \ge 0$ then $z(t) \ge 0$ for any $t\in [t_{0},t_{0} + T^{*}(t_{0},z_{0}))$.
\end{corollary}
\begin{proof}
Consider $z_{0,n} = z_{0}+1/n$, for any $0<T<T^{*}(t_{0},z_{0})$, there exists $n_{0}\in\mathbb{N}$
such that $T<T^{*}(t_{0},z_{0,n})$ if $n\ge n_{0}$. Since $z_{0,n}>0$, using Proposition \ref{pr: u>0}
we have $z_{n}(t)>0$ for $t\in [t_{0},t_{0} + T]$.
As $z_{n}$ converges to $z$ in $C( \Theta \times[t_{0},t_{0} + T])$, we see that $z\ge 0$.
Since $T$ is arbitrary, we obtain the result.
\end{proof}
We now show global well-posedness in $C(\Theta)$ for $z_{0}\ge 0$.
\begin{proposition}
If $z_{0} \in C(\Theta)$ with $z_{0}\ge 0$, then $T^{*}(t_{0},z_{0}) = \infty$.
\end{proposition}

\begin{proof}
Let $0<T<T^{*}(t_{0},z_{0})$, from Corollary \ref{co: uge0}, we obtain that $a(t,\theta),g(t,\theta)\ge 0$ and then $A(t,t',\theta) \le \|k\|_{T,\infty} (t-t')$. 
Integrating \eqref{eq: z(t)} on $ \Theta $, we get for $t\in [t_{0},t_{0} + T]$
\begin{align*}
\int_{\Theta}z(t,\theta)d\mu(\theta) & \le \exp(\|k\|_{T,\infty} (t-t_{0}))\int_{\Theta} z_{0}(\theta)d\mu(\theta) \\
& + \int_{t_{0}}^{t}\int_{\Theta\times \Theta}\exp(\|k\|_{T,\infty} (t-t'))
M(t',\theta,\vartheta )z(t',\vartheta) d\mu(\vartheta) d\mu(\theta) dt' \\
& \le \exp(\|k\|_{T,\infty} (t-t_{0}))\left(\int_{\Theta} z_{0}(\theta)d\mu(\theta) + \right.\\
& \left. + \|M\|_{T,\infty}\int_{0}^{t}\int_{\Theta}\exp(-\|k\|_{T,\infty} (t'-t_{0}))z(t',\vartheta) d\mu(\vartheta) dt'\right),
\end{align*}
using Gronwall's lemma, we obtain
\begin{align*}
\int_{\Theta}z(t,\theta)d\mu(\theta) & \le 
\exp((\|k\|_{T,\infty} + \|M\|_{T,\infty}) (t-t_{0}))\int_{\Theta} z_{0}(\theta)d\mu(\theta) \\
& \le \exp((\|k\|_{T,\infty} + \|M\|_{T,\infty}) (t-t_{0}))|z_{0}|_{\infty},
\end{align*}
which implies $0 \le g(t,\theta)
\le \|M\|_{T,\infty} \exp((\|k\|_{T,\infty} + \|M\|_{T,\infty}) (t-t_{0}))|z_{0}|_{\infty}$.
From \eqref{eq: z(t)}, we get
\begin{align*}
\|z\|_{T,\infty}\le \exp((\|k\|_{T,\infty} + \|M\|_{T,\infty}) T) |z_{0}|_{\infty}.
\end{align*}
And finally we have that $T^{*}(z_{0}) = \infty$.
\end{proof}
Now, we construct a positive $F$--invariant convex set of $Z$.
\begin{lemma}
\label{le: w}
Let $w\in C^{1}([t_{0},t_{0}+T])$, $w \ge 0$, 
such that $\dot{w}\le k w - c w^{2}$, with $k,c>0$.
If $\lambda\ge k/c$ and $0\le w(t_{0})\le \lambda$,
then $0\le w(t) \le \lambda$ for $t\in [t_{0},t_{0}+T]$.
\end{lemma}
\begin{proof}
Suppose $w(t_{+})>\lambda$ with $t_{0} < t_{+}\le t_{0}+T$, consider 
$t_{-}=\sup\{t\in[t_{0},t_{+}]: w(t)\le \lambda\}$. Using the mean value theorem,
there exists $t_{1} \in (t_{-},t_{+})$ such that
\begin{align*}
w(t_{+})-w(t_{-}) = \dot{w}(t_{1})(t_{+}-t_{-}),
\end{align*}
then $\dot{w}(t_{1})>0$. But $w(t_{1})>\lambda$, which implies $ k w(t_{1}) - c w^{2}(t_{1})<0$, a contradiction.
\end{proof}
\begin{proposition}
\label{pr: C in Kl}
Let $z_{0}\in C( \Theta )$, $z_{0}\ge 0$. If $\lambda(t) \ge \max\{k_{+}(t)/c_{-}(t),|z_{0}|_{Z}\}$, then
the solution of \eqref{eq: integral equation} $z\in C([t_{0},\infty),C(\Theta))$ verifies
$z(t)\ge 0$ and $|z(t)|_{Z}\le\lambda(t)$ for any $t\ge t_{0}$. 
\end{proposition}
\begin{proof}
From Corollary \ref{co: uge0}, we can see that $z(t) \ge 0$.
Let $t>0$, for any $t'\in[t_{0},t]$ we have
\begin{align*}
\frac{d}{dt}\int_{\Theta}z(t',\theta)d\mu(\theta) & = \int_{\Theta}k(t',\theta)z(t',\theta)d\mu(\theta)
+ \int_{\Theta\times \Theta}M(t',\theta,\vartheta )z(t',\vartheta) d\mu(\vartheta) d\mu(\theta) \\
& - \int_{\Theta\times \Theta}C(t',\theta,\vartheta )z(t',\vartheta)z(t',\theta) d\mu(\vartheta) d\mu(\theta)\\
& = \int_{\Theta}\left(k(t',\theta) + \int_{\Theta} M(t',\vartheta ,\theta) d\mu(\vartheta) \right) z(t',\theta) d\mu(\theta) \\
& - \int_{\Theta\times \Theta}C(t',\theta,\vartheta )z(t',\vartheta)z(t',\theta) d\mu(\vartheta) d\mu(\theta).
\end{align*}
From \eqref{eq: kc}, we have
\begin{align*}
\frac{d}{dt}\int_{\Theta}z(t',\theta)d\mu(\theta) & \le k_{+}(t) \int_{\Theta}z(t',\theta)d\mu(\theta) 
- c_{-}(t) \left( \int_{\Theta}z(t',\theta)d\mu(\theta) \right)^{2}.
\end{align*}
Using Lemma \ref{le: w}, we obatin $|z(t)|_{Z}\le \lambda(t)$.
\end{proof}
\begin{proposition}
\label{pr: Kl}
Let $\lambda\in C\pa{\mathbb{R}_{+}}$ be an increasing function such that $\lambda(t)\ge k_{+}(t)/c_{-}(t)$.
Then, the family of bounded convex closed set $\{K(t)\}_{t\in\mathbb{R}_{+}}$ given by
$K\pa{t} = \{z\in Z: z\ge 0 \text{ a.e.}, |z|_{Z}\le\lambda(t)\}$ is increasing and positive $F$-invariant.
\end{proposition}
\begin{proof}
Let $z_{0}\in K(t_{0})$, taking $\{z_{0,n}\}_{n\in\mathbb{N}}\subset C( \Theta )\cap K(t_{0})$ 
such that $|z_{0} - z_{0,n}|_{Z} \to 0$,
from Proposition \ref{pr: C in Kl} we can see that $T^{*}(z_{0,n})=\infty$ and $z_{n}(t)\in K(t)$, for $t\ge t_{0}$.
Using continuous dependence on initial data, 
we can see that $|z(t) - z_{n}(t)|_{Z} \to 0$ for any $t\in[t_{0},t_{0} + T^{*}(t_{0},z_{0}))$,
since $K(t)$ is closed, we obtain $z(t)\in K(t)$.
\end{proof}
\begin{remark}
Also, the familiy $\{K(t)\}_{t\in\mathbb{R}_{+}}$ is positive $2F$-invariant.
\end{remark}

\begin{theorem}
Let $u_{0}\in C_{\rm u}(\mathbb{R}^{d},Z)$, with $u_{0}(x)\ge 0$ a.e. in $\Theta$, the mild solution of equation 
\eqref{eq: in L1(X)} is globally well-posed and verifies $\|u(t)\|_{\infty,Z}\le\max\{\|u_{0}\|_{\infty,Z}, k_{+}(t)/c_{-}(t)\}$.
\end{theorem}
\begin{proof}
The result is a immediate consequence of Theorem \ref{th: global existence} and Proposition \ref{pr: Kl}
taking $\lambda(t) = \max\{\|u_{0}\|_{\infty,Z}, k_{+}(t)/c_{-}(t)\}$.
\end{proof}
\subsection{Global existence for products of Banach Spaces}
We generalize the previous results by proving global existence for products of Banach Spaces. Lemma \ref{le: convex product st} proves that the semigroup operator maintains the solution inside the invariant region. Following that, Theorem \ref{th: global existence2} proves that if $u_0$ is inside the invariant region, then $u(t)$ remains in it for all $t>0$.
Let $\{Z_{j}\}_{1\le j\le m}$ be Banach spaces and
$Z=Z_{1}\times \cdots \times Z_{m}$ with the usual norm,
we denote $\pi_{j}:Z\to Z_{j}$ the projection map. 
If $\sigma_{j}>0$, $0<\beta_{j}\le 1$, and ${\sf S}_{j}(t)u=G_{\sigma_{j},\beta{j}}(.,t)*u$
for $u\in C_{u}(\mathbb{R}^{d},Z_{j})$, then ${\sf S}:\mathbb{R}_{+}\to \mathcal{B}(C_{u}(\mathbb{R}^{d},Z))$ given by
\begin{align*}
{\sf S}(t)u=({\sf S}_{1}(t)\pi_{1}u,\ldots,{\sf S}_{m}(t)\pi_{m}u)
\end{align*}
is a continuous contraction semigroup.
\begin{lemma}
\label{le: convex product st}
Let $K_{j}\subset Z_{j}$ be a closed convex set and $K = K_{1}\times\cdots\times K_{m}\subset Z$.
If $u\in C_{u}(\mathbb{R}^{d},K)$, then ${\sf S}(t)u \in C_{u}(\mathbb{R}^{d},K)$, for any $t>0$.
\end{lemma}
\begin{proof}
The proof is a consequence of the definition above and lemma \ref{le: G*u}.
\end{proof}
\begin{theorem}
\label{th: global existence2}
Let $K_{j}(t)\subset Z_{j}$ be bounded closed convex sets,
if $K(t) = K_{1}(t)\times\cdots\times K_{m}(t)$ is a positively $F$--invariant set, 
then for any $u_{0}\in C_{u}(\mathbb{R}^{d},K(0))$,
it holds $T^{*}(u_{0})=\infty$ and $u(t)\in C_{u}(\mathbb{R}^{d},K(t))$ for $t>0$.
\end{theorem}
\begin{proof}
	
Let $u_{0}\in C_{u}(\mathbb{R}^{d},K(0))$ and $T^{*}(u_{0})$ 
maximal time of existence of the solution $u$ of \eqref{eq: mild solution}. 
Let $t\in (0,T^{*}(u_{0}))$, $h = t/n$, $n\in\mathbb{N}$, $\{V_{h,k}\}_{1\le k\le n}$ and $\{U_{h,k}\}_{0\leq k\leq n}$, as defined in Proposition \ref{pr: u in CK}.
Suppose that $U_{h,k}\in C_{u}(\mathbb{R}^{d},K(kh))$. Lemma \ref{le: convex product st} implies that
$V_{h,k+1}\in C_{u}(\mathbb{R}^{d},K(kh))$.
Using that $K(t) = K_{1}(t)\times\cdots\times K_{m}(t)$ is positively $F$--invariant, $U_{h,k+1}\in C_{u}(\mathbb{R}^{d},K((k+1)h))$.
Using the same reasoning as in Proposition \ref{pr: u in CK}, we have that $U_{h,k}\to u(kh)$ in $C_{u}(\mathbb{R}^{d},Z)$ when $n\to\infty$ and $u(t)\in C_{u}(\mathbb{R}^{d},K(t))$.
Since $K(t)$ is bounded, we obtain the result.
\end{proof}
\begin{example}
We expose an example, where we construct an invariant convex set that consists of a product of intervals, in which we can apply the above results. In \cite{Asgari2011} a FHN Model for pattern formation is presented:
\begin{equation}
\label{eq: FHN}
\begin{cases}
\partial_{t} u=\sigma_{u}\Delta u+(a-u)(u-1)u-v \\
\partial_{t} v=\sigma_{v}\Delta v+e(bu-v) 
\end{cases}
\end{equation}
with $0<a<1$, $e>0$ and $b\ge 0$.
A similar example is analyzed in \cite{Smoller1983}.
To apply Theorem \ref{th: global existence2}, we need to find positive $F$-invariant rectangle
$K=K_{1}\times K_{2}$, $K_{j}=[-R_{j},R_{j}]$, where $F$ is given by 
\begin{equation*}
F(u,v)=(au^{2}-u^{3}-au+u^{2}-v,e(bu-v)).
\end{equation*}
Let $R_{1}>\max\{4,\sqrt{2b}\}$ and $2bR_{1} < 2R_{2} < R_{1}^{3}$,
we can see that the rectangle with $R_{1}$ and $R_{2}$ is $F$-invariant:
\begin{align*}
F_{1}(R_{1},v) & \le a(R_{1}^{2}-R_{1})-R_{1}^{3}+R_{1}^{2}+|v|
\le a(R_{1}^{2}-R_{1})-R_{1}^{3}+R_{1}^{2}+R_{2} < 0, \\
F_{1}(-R_{1},v) & \ge a(R_{1}^{2}+R_{1})+R_{1}^{3}+R_{1}^{2}-|v| 
\ge a(R_{1}^{2}+R_{1})+R_{1}^{3}+R_{1}^{2}-R_{2} >0, \\
F_{2}(u,R_{2}) & \le e(b|u|-R_{2}) \le e(bR_{1}-R_{2})<0, \\
F_{2}(u,-R_{2}) & \ge e(-b|u|+R_{2}) \ge e(-bR_{1}+R_{2})>0.
\end{align*}
Then the field evaluated at the border of $K$ points inward. 
By Theorem \ref{th: global existence2} the equation \eqref{eq: FHN} is globally well posed.
\end{example}
\section{Asymptotic behavior}
We analyze the situation in which, if $u_0$ has a horizontal asymptote at $z_0$ then using the same methods as before, 
we prove that $u(t)$ approaches asymptotically to the time evolution of $z_0$. We consider the 1-dimensional real case. 
We first show in Lemma \ref{le: lim Su} that if $u_0$ has a horizontal asymptote at $z_0$ then ${\sf S}(t)u_0$ 
remains with the same horizontal asymptote. Next, we prove in Lemma \ref{le: lim Nu} that $\mathsf{N}(t,t_0,u_{0})(x)$ 
has a time dependent horizontal asymptote, which is the solution of the equation \eqref{eq: integral equation} with $z_0$ as an initial condition. 
Finally, we combine both results and a continuous dependence argument in Lemma \ref{le: lim un} 
to achieve Proposition \ref{pr: lim u}, the solution $u(t)$ of \eqref{eq: reaction-diffusion} has the same time dependent 
horizontal asymptote $z(t)$.

These results can be applied, for example, to the Fisher-Kolmogorov equation.
Specifically, in \cite{Kolmogorov1989} solutions with the mentioned asymptotic behavior are analyzed.
\\
\begin{proposition}
	\label{pr: lim u}
	Let $u_{0}\in C_{\rm u}(\mathbb{R},Z)$ such that
	$\lim_{x\to\pm\infty}u_{0}(x)=z_{0}^{\pm}\in Z$, if $u(t)$ is the solution
	of \eqref{eq: mild solution} with $F$ autonomous, then $\lim_{x\to\pm\infty}u(t,x)=z^{\pm}(t)$, where
	$z^{\pm}$ is the solution of \eqref{eq: integral equation} with $z^{\pm}(0) = z_{0}^{\pm}$.
\end{proposition}
\begin{lemma}
	\label{le: lim Su}
	Let $u_{0}\in C_{\rm u}(\mathbb{R},Z)$ such that
	$\lim_{x\to\pm\infty}u_{0}(x)=z_{0}^{\pm}\in Z$. If $u(t)={\sf S}(t)u_{0}$,
	then $\lim_{x\to\pm\infty}u(t,x)=z_{0}^{\pm}$.
\end{lemma}
\begin{proof}
	We only prove for $z_0^+$, the $z_0^-$ case is similar. Let $\varepsilon>0$, there exists $x_{*}^{+}>0$ such that $|u_{0}(x)-z_{0}^{+}|_{Z}<\varepsilon$ for $x>x_{*}^{+}$.
	Before proving the limit, we need an estimate of $g_{\beta}(\xi)$. Taking $r>0$ large enough, we have
	\begin{align}
	\label{eq: gbbound}
	\int_{|\xi|> (\sigma t)^{-1/(2\beta)}r}g_{\beta}(\xi)d\xi<\varepsilon/(2|u_{0}|_{\infty,Z}),
	\end{align}
	Next, to study the asymptotic convergence, we analyze two cases, 
	if $x>x_{*}+r$ then,
	\begin{align*}
	|u(t,x)-z_{0}^{+}| & \le \int_{\mathbb{R}}G_{\sigma,\beta}(t,x-y)|u_{0}(y)-z_{0}^{+}|dy \\
	& = \int_{y>x-r}G_{\sigma,\beta}(t,x-y)|u_{0}(y)-z_{0}^{+}|dy \\
	& + \int_{y<x-r}G_{\sigma,\beta}(t,x-y)|u_{0}(y)-z_{0}^{+}|dy = I_{1} + I_{2}.
	\end{align*}
	Since $y>x-r>x_{*}^{+}$, we have that $|u_{0}(y)-z_{0}^{+}|<\varepsilon$ and therefore we can bound the first integral,
	\begin{align*}
	I_{1} \le \varepsilon \int_{\mathbb{R}}G_{\sigma,\beta}(t,x-y)dy = \varepsilon.
	\end{align*}
	For the second integral, we will use estimate \eqref{eq: gbbound}, and the norm of the initial condition $u_0$,
	\begin{align*}
	I_{2} & \le 2|u_{0}|_{\infty,Z}\int_{y<x-r}G_{\sigma,\beta}(t,x-y)dy = 
	2|u_{0}|_{\infty,Z}\int_{\xi>r}G_{\sigma,\beta}(t,\xi)d\xi \\
	& = 2|u_{0}|_{\infty,Z}\int_{|\xi'|> (\sigma t)^{-1/(2\beta)}r}g_{\beta}(\xi')d\xi'
	< \varepsilon
	\end{align*}
	Bounding both integrals we prove the result.
\end{proof}
\begin{lemma}
	\label{le: lim Nu}
	Let $u_{0}\in C_{\rm u}(\mathbb{R},Z)$ such that
	$\lim_{x\to\pm\infty}u_{0}(x)=z_{0}^{\pm}\in Z$. If $u(t)={\sf N}(t,t_0,u_{0})$,
	then $\lim_{x\to\pm\infty}u(t,x) = z^{\pm}(t)$, where $z^{\pm}(t)$ is the solution
	\eqref{eq: integral equation} with $z^{\pm}(0) = z_{0}^{\pm}$.
\end{lemma}
\begin{proof}
	We again consider only the $z^+$ case. 
	From continuous dependence of the initial data, for $\varepsilon>0$, there exists $\delta>0$ such that if $|z_{0}^{+}-z_{0}|_{Z}<\delta$, then $|z^{+}(t) - z(t)|_{Z}<\varepsilon$.
	Let $x_{*}^{+}\in\mathbb{R}$ such that if $x>x_{*}^{+}$, 
	$|u_{0}(x)-z_{0}^{+}|_{Z}<\delta$, then $|u(t,x)-z^{+}(t)|_{Z}<\varepsilon$.
\end{proof}
\begin{lemma}
	\label{le: lim un}
	Let $\{u_{n}\}_{n\in\mathbb{N}}\subset C_{\rm u}(\mathbb{R}^{d},Z)$ such that $u_{n}\to u$
	in $C_{\rm u}(\mathbb{R}^{d},Z)$. If for $n\in\mathbb{N}$, it holds $\lim_{x\to\pm}u_{n}(x)=z^{\pm}$, then $\lim_{x\to\pm}u(x)=z^{\pm}$.
\end{lemma}
\begin{proof}
	Let $\varepsilon>0$, we can take $n\in\mathbb{N}$ such that 
	$\|u-u_{n}\|_{\infty,Z}<\varepsilon/2$.
	Then there exists $x_{*}^{+}\in\mathbb{R}$ such that
	$|u_{n}(x)-z^{+}|_{Z}<\varepsilon/2$ if $x>x_{*}^{+}$. Therefore, 
	\begin{align*}
	|u(x)-z^{+}|_{Z} \le |u(x)-u_{n}(x)|_{Z} + |u_{n}(x)-z^{+}|_{Z} < \varepsilon.
	\end{align*}
\end{proof}
\begin{proof}\textbf{(Proof of Proposition \ref{pr: lim u})}
	Let $n\in\mathbb{N}$, $h=t/n$ and $\{U_{h,k}\}_{0\le k\le n}$, $\{V_{h,k}\}_{1\le k\le n}$ the sequences defined by \eqref{eq: L-T}.
	We claim that $\lim_{x\to\pm\infty}U_{h,k}(x) = z^{\pm}(kh)$ for $k=0,\dots,n$.
	Clearly, the assertion is true for $k=0$.
	If $\lim_{x\to\pm\infty}U_{h,k}(x) = z^{\pm}(kh)$, from Lemma \ref{le: lim Su}, we see that
	$\lim_{x\to\pm\infty} V_{h,k+1} = z^{\pm}(kh)$, 
	and using Lemma \ref{le: lim Nu} we obtain $\lim_{x\to\pm\infty}U_{h,k+1}(x)$.
	We conclude $z^{\pm}(t)=z^{\pm}(nh)=\lim_{x\to\pm\infty}U_{h,n}(x)$ and, since $U_{h,n}\to u(t)$,
	Lemma \ref{le: lim un} implies the result.
\end{proof}

\section*{Acknowledgement}
This work was partially supported by CONICET--Argentina, PIP 11220130100006.

\end{document}